\documentclass[12pt]{article}
\usepackage{amsfonts}
\usepackage{amsmath,amsthm}
\usepackage{cases}

\newtheorem{lemma}{Lemma}[section]
\newtheorem{theorem}{Theorem}[section]

\theoremstyle{definition}

\newtheorem{remark}{Remark}

\newcommand{\comment}[1]{}

\numberwithin{equation}{section}
\begin{document}
\title{Existence of Solutions to Mean Field Equations on Graphs}
\author{An Huang,Yong Lin,Shing-Tung Yau}
\date{}
\maketitle

\renewcommand{\thefootnote}{\fnsymbol{footnote}}

\begin{center}
\textbf{Abstract}
\end{center}
In this paper,  we prove two existence results of solutions to mean field equations
$$\Delta u+e^u=\rho\delta_0$$
and
$$\Delta u=\lambda e^u(e^u-1)+4 \pi \sum_{j=1}^{M}{\delta_{p_j}}$$

on an arbitrary connected finite graph, where $\rho>0$ and $\lambda>0$ are constants, $M$ is a positive integer, and $p_1,...,p_M$ are arbitrarily chosen vertices on the graph.

\section{Introduction}

The mean field equation
\begin{equation}\label{MFE}
	\Delta u+e^u=\rho\delta_0
\end{equation}\label{equ:1}has its origin in the prescribed curvature problem in geometry. Closely related is the Kardan-Warner equation \cite{KW74}
 \begin{equation}\label{KW}
 	\Delta u+h e^u=c.
 \end{equation}\label{equ:2}

The name of the equation \eqref{MFE} comes from statistical physics as the mean field limits of the Euler flow \cite{CLMP92}. It has also been shown to be related to the Chern-Simons-Higgs model. The existence of solutions to equation \eqref{MFE} has been studied in \cite{CL03}, \cite{CLW04}, \cite{LW10}, \cite{LY13} on Euclidean spaces and on the two dimensional flat tori. For example, on the two dimensional flat tori, when $\rho \neq 8m \pi $ for any $m\in \mathbf{Z}$, equation \eqref{MFE} always has solutions, see \cite{CL03}, \cite{CLW04}. When $\rho =8 \pi $, it was shown in \cite{LW10} that equation \eqref{MFE} has solutions if and only if the Green's function on the two dimensional flat tori has critical points other than the three half period points.

In \cite{GLY16}, Grigorigan, Lin and Yang have obtained a few sufficient conditions when equation \eqref{KW} has a solution on a finite graph. There are several further results regarding the solutions of \eqref{KW} on graphs in \cite{G17}, \cite{GJ17}, \cite{KS18}.

In this paper, we study equation \eqref{MFE} and also the following mean field equation on graphs:

\begin{equation}\label{equ:4}
	\Delta u=\lambda e^u(e^u-1)+4 \pi \sum_{j=1}^{M}{\delta_{p_j}},
\end{equation} where $\lambda >0$, $M$ is any fixed positive integer, and $p_1,...,p_M$ are arbitrarily chosen vertices on the graph.

Caffarelli and Yang in \cite{CY95} proved an existence result of solutions to equation \eqref{equ:4} on doubly periodic regions in $\mathbf{R}^2$ (the 2-tori), depending on the value of the parameter $\lambda$.

In this paper, we show that equation \eqref{MFE} always has a solution on any connected finite graph (Theorem \ref{2.1}), in contrast to the continuous case.
We shall also prove an existence result for equation \eqref{equ:4} on a connected finite graph (Theorem \ref{2.2}), depending on the value of the parameter $\lambda$, which is in line with the result of Caffarelli and Yang on the 2-tori.

We obtain these results by a mostly straightforward adaption of existing treatments from the continuous case \cite{KW74}, \cite{CY95}, \cite{GLY16}. Once we have the setup, some analysis tend to simplify on finite graphs since there is only a finite number of degrees of freedom. Theorem \ref{2.1} on the other hand shows that the existence of solutions for \eqref{MFE} on the discrete two dimensional tori graph given as the quotient of the two dimensional infinite lattice graph by a rank 2 sublattice, differs from that on the continuous limit-- the two dimensional flat tori, when the parameter $\rho$ takes on certain special values such as $8\pi$.

\begin{remark}
As a side remark, it appears interesting to study the Green's function on the 2-tori by studying the corresponding discrete Green's function on the 2-tori graph stated above. For example, when the torus parameter $\tau=\frac{1}{2}+i$, there exist two additional critical points of the Green's function besides the half periods by \cite{LW10}. A computer study aided by this discrete Green's function indicates that the slope of the line through these two additional critical points of the Green's function is equal to $\frac{25}{64}$.
\end{remark}

\section{Settings and main results}

Let $G=(V,E)$ be a connected finite graph, where $V$ is the set of vertices and $E$ is the set of edges. Denote $N=|V|$. We allow positive symmetric weights $w_{xy}=\omega_{yx}$ on edges $xy \in E$. Let $\mu :V \rightarrow \mathbb{R^{+}} $ be a finite measure. For any function $u:V \rightarrow \mathbb{R}$, the Laplace operator acting on $u$ is defined by $$\Delta u(x)=\frac{1}{\mu (x)}\sum_{y \sim x}{w_{xy}(u(y)-u(x))}, $$where $ y \sim x$ means $xy \in E $. The gradient form of $u$ is by definition $$\Gamma (u) =\frac{1}{2}\int_{V}{|\bigtriangledown u|^2} :=\sum_{x\in V}\frac{1}{2 \mu (x)}\sum_{y \sim x}{w_{xy}(u(y)-u(x))^2} .$$We use the notation $ \int_{V}{f(x)d \mu (x)=\sum_{x \epsilon V}{f(x)\mu (x)}}$. As in \cite{GLY16}, we define a Sobolev Space and a norm by $$W^{1,2}(V)=\{u:V \rightarrow \mathbb{R}:\int_{V}(\vert \bigtriangledown u \vert^2+u^2)d \mu <+\infty\}$$and $$\vert \vert u \vert \vert _{W^{1,2}(V)}=(\int_{V}(\vert \bigtriangledown u \vert^2+u^2)d \mu )^{1/2} $$respectively. Since V is a finite graph, $ W^{1,2}(V)$ is $V^{\mathbb{R}} $, the finite dimensional vector space of all real functions on $V$. We have the following Sobolev embedding (Lemma 5 in \cite{GLY16}):

\begin{lemma}\label{lemma2.1}
	Let $G=(V,E)$ be a finite graph. The Sobolev Space $ W^{1,2}(V)$ is precompact. Namely, if $\{u_j\}$ is bounded in $W^{1,2}(V)$, then there exits some $u \in W^{1,2}(V) $ such that up to a subsequence, $ u_j \rightarrow u$ in $ W^{1,2}(V)$.
\end{lemma}

\begin{remark}
For finite graphs, Lemma \ref{lemma2.1} can be avoided for the purpose of the present paper. But we include it for potential generalizations to infinite graphs.
\end{remark}

By using the variational principle (see the similar approach in \cite{KW74} and \cite{GLY16}) , we prove the following
\begin{theorem}\label{2.1}
	Equation \eqref{MFE} has a solution on $G$.
\end{theorem}

Using an iteration method, we next prove the following

\begin{theorem}\label{2.2}
	There is a critical value $\lambda_c$ depending on $G$ satisfying
$$\lambda_c\ge \frac{16\pi M}{|V|},$$ such that when
$\lambda> \lambda_c$,
the equation \eqref{equ:4} has a solution on $G$, and when $\lambda< \lambda_c$, the equation \eqref{equ:4} has no solution.
\end{theorem}

\section{The proof of Theorem \ref{2.1}}
\begin{proof}
For $u \in W^{1,2}(V) $, we consider the functional $$J(u)=\frac{1}{2}\int_{V}{|\bigtriangledown u|^2} + \int_{V} \rho \cdot \delta_0\cdot u.$$Let the set $$B=\{ u \in W^{1,2}(V) :\int_{V}{e^u}=\int_{V}{\rho\cdot \delta_0 } =\rho\} .$$
First we verify $B \neq \emptyset$: let $l>0$, define
\[u_l(x)=
\left\{
\begin{array}{ll}
e^l, & \hbox{$x=x_0$} \\
0, & \hbox{otherwise}
\end{array}
\right.
\]

and

\[\widetilde{u_l}(x)=
\left\{
\begin{array}{ll}
e^l, & \hbox{$x=x_0$} \\
0, & \hbox{otherwise}
\end{array}
\right.
\]

then $$\int_{V}{e^{u_l}=e^l \rightarrow +\infty \quad (l \rightarrow +\infty)},$$
and $$\int_{V}{e^{\widetilde{u_l}}}=e^{-l} \rightarrow 0 \quad(l \rightarrow +\infty) ,$$
Let $\Phi (t)=\int_{V}{e^{tu_l+(1-t)\widetilde{u_l}}} ,$ then for sufficiently big $l$, $$\Phi (0)<\rho<\Phi (1),$$

so there exists $t \in (0,1)$ such that $\Phi (t)=\rho$, therefore $B \neq \emptyset.$

For any $u \in B,$ $\int_{V}{e^u}=\rho$, choose $x_D \in V$ such that
$$e^{u(x_D)}=\min_{x\in V}{\{e^{u(x)}\}},$$ then
 $$Ne^{u(x_D)} \leq \rho,$$
 $$u(x_D) \leq \log \frac{\rho}{N}.$$
 Choose a shortest path on $G$ from $x_0$ to $x_D$ (therefore non-backtracking): $x_0 \sim x_1 \sim ...x_{D-1} \sim x_D,$ fix any $0<\epsilon<1$,
\[
\begin{split}
\frac{1}{2} \int_{V}{|\bigtriangledown u|^2}&=\sum_{x\in V}{\frac{1}{2 \mu (x)}} \sum_{y \sim x}{\omega_{xy}(u(y)-u(x))^2}\\
&\geq D_{eg} [(u(x_1)-u(x_0))^2+(u(x_2)-u(x_1))^2+...+(u(x_D)-u(x_{D-1}))^2]\\
&\geq D_{eg} \frac{(u(x_0)-u(x_D))^2}{D}\\
&\geq \frac{D_{eg}}{D}\cdot (u(x_0) -\log \frac{\rho}{N})^2\\
%&\geq M_1\cdot |u(x_0)|.
\end{split}.
\]
where $D_{eg}= \min_{x\in V,y\sim x}\frac{\omega_{xy}}{2\mu(x)}$,

So there exists $c>0$ depending on $\epsilon$, such that when $u(x_0) \geq c$, $\frac{1}{2} \int_{V}{|\bigtriangledown u|^2}\geq M_1\cdot |u(x_0)|$ for some $M_1 \geq \frac{\rho_2}{\epsilon}$. Therefore we have in this case
\begin{equation}\label{3.1}
	J(u) \geq (1-\epsilon)\cdot \frac{1}{2} \int_{V}{\vert \bigtriangledown u \vert^2}.
\end{equation}When $|u(x_0)| <c,$
 \begin{equation}\label{3.2}
 	J(u) > \frac{1}{2}\int_{V}{\vert \bigtriangledown u \vert^2-\rho c}.
 \end{equation}

 Therefore $J(u)$ has a lower bound on $B$. So we can choose
 $$u_k(x) \in B,J(u_k(x))\rightarrow b \quad (k\rightarrow \infty),$$ where $b=\inf_{u \in B}{J(u)}.$

 From \eqref{3.1} and \eqref{3.2}, for all $k$,
 $$\int_{V}{|\bigtriangledown u_k |^2}\leq c_1$$ for some constant $c_1$,
 since $|J(u_k) |\leq c_2$ for some constant $c_2$.
 As $$J(u_k)=\frac{1}{2} \int_{V}{\vert \bigtriangledown u_k \vert^2+\rho\cdot u_k(x_0)},$$ \\ there exists a constant $c'$, such that $|u_k(x_0)| \leq c'$ for all $k$ . For any $x \in V$, choose a shortest path on G from $x_0$ to $x$: $$x_0 \sim x_1 \sim ...x_{D'-1} \sim x,$$
\[
\begin{split}
|u_k(x) | &\leq |u_k(x)-u_k(x_{D'-1})|+|u_k(x_{D'-1})-u_k(x_{D'-2})| +...+| u_k(x_1)-u_k(x_0) | +|u_k(x_0) |\\
&\leq D\cdot [| u_k(x)-u_k(x_{D'-1}) |^2 +...+| u_k(x_1)-u_k(x_0) |^2 ]^{1/2}+c'\\
&\leq \frac{D'}{D_{eg}} \int_{V}| \bigtriangledown u_k |^2+c'\\
%&\leq c.
\end{split}.
\]

As $D'\leq N$, the $L^{\infty}$ norm of $u_k(x)$ is uniformly bounded, and therefore $u_k(x)$ are uniformly bounded in $W^{1,2}(V)$. From the Sobolev embedding (Lemma \ref{lemma2.1}),
there exits a subsequence $u_{k1}(x) \rightarrow u_{\infty}(x)\in W^{1,2}(V)$ in $W^{1,2}(V)$, and
$$\int_{V}{e^{u_{\infty}}}=\lim\limits_{k_1 \rightarrow \infty} \int_{V}{e^{u_{k1}}}=\rho.$$
Finally we prove that $u_{\infty}$ is the solution of equation \eqref{MFE}. This is based on the method of Lagrange maltiplies. Let
$$L(t,\lambda)=\frac{1}{2} \int_{V} \vert \bigtriangledown(u_{\infty}+t\varphi) \vert^2+\int_{V}{\rho\cdot \delta_0(u_{\infty}+t\varphi)}+\lambda (-\int_{V}{e^{u_{\infty}+t\varphi}}+\rho),$$
where $\varphi \in W^{1,2}(V).$  So we have
$$\frac{\partial L}{\partial \lambda}\mid_{t=0}=-\int_{V}{e^{u_{\infty}}+\rho}=0,$$
since $u_{\infty} \in B$. And
$$0=\frac{\partial L}{\partial t}|_{t=0}=-\int{\Delta u_{\infty} \cdot \varphi}+\int{\rho \cdot \delta_0 \cdot \varphi}-\lambda \int{e^{u_{\infty}} \cdot \varphi}=0.$$
Therefore by the variational principle, $$-\Delta u_{\infty}+\rho \cdot \delta_0-\lambda \cdot e^{u_{\infty}}=0.$$
 Since $\int_{V}{\Delta u_{\infty}}=0$, we have
 $$\lambda \int_{V}{e^{u_{\infty}}}=\int_{V}{\rho\cdot \delta_0}=\rho. $$ \\ So $\lambda =1$, and $$\Delta u_{\infty}+e^{u_{\infty}}=\rho \cdot \delta_0.$$
This finishes the proof of Theorem \ref{2.1}.
\end{proof}
\section{The proof of Theorem \ref{2.2}}

We use the method of upper and lower solutions to prove Theorem \ref{2.2}, adapting methods from \cite{KW74}, \cite{CY95} and \cite{GLY16} to the graph setting.\\
%The following series of Lemmas give the proof of the Theorem 2.2.
\begin{lemma}\label{lemma4.1}
(Maximum principle) Let $G=(V,E)$, where $V$ is a finite set, and $K\ge 0$ is a constant. Suppose a real function $u(x): V\to \mathbb{R}$ satisfies $$(\Delta -K)u(x)\geq 0 \text{ for all } x\in V,$$ \\ then $u(x)\leq 0$ for all $x \in V$.
\end{lemma}

\begin{proof}Let $u(x_0)=\max_{x\in V}{\{u(x)\}}$, we only need to show that $u(x_0)\leq 0$. Suppose this is not the case. Since $$(\Delta -K)u(x_0)\geq 0,$$ \\ we have $$\sum_{y \sim x}{u(y)}\geq (d_{x_0}+K)u(x_0)\geq d_{x_0}u(x_0),$$ \\ where we have used the assumption that $u(x_0)> 0$, and that $K\ge 0$ in the last inequality. This implies that for any $y\sim x,u(y)\geq u(x_0)$. Since $G$ is a connected graph, by induction, for any $xy\in E,u(y)=u(x_0)$. From $$K \int_{V}u(x)\leq \int{\Delta u(x)}=0$$ and $K\ge 0$ we get that $u(x_0)\leq 0$. This is a contradiction.
\end{proof}

Let $u_0$ be a solution of the Poisson equation
\begin{equation}\label{4.1}
	\Delta u_0=-\frac{4\pi M}{|V|}+4\pi \sum_{j=1}^{M}\delta_{p_j}.
\end{equation}
The solution of \eqref{4.1} always exists, as the integral of the right side is equal to 0. Inserting $u=u_0+v$ into equation \eqref{equ:4}, we get
\begin{equation}\label{4.2}
	\Delta v=\lambda e^{u_0+v}(e^{u_0+v}-1)+\frac{4\pi M}{|V|}.
\end{equation}

Sum the two sides of the about equation, we get

$$\lambda (e^{u_0+v}-\frac{1}{2})^2=\frac{\lambda}{4}-\frac{4\pi M}{|V|},$$

which implies that
\begin{equation}\label{4.3}
\lambda\ge \frac{16\pi M}{|V|}.
\end{equation}

We call a function $v_+$ an upper solution of \eqref{4.2} if for any $x\in V$, it satisfies
\begin{equation}\label{4.4}
	\Delta v_+(x)\geq \lambda e^{u_0(x)+v_+(x)}(e^{u_0(x)+v_+(x)}-1)+\frac{4\pi M}{|V|}.
\end{equation}
Let $v_0=-u_0$, we define a sequence $\{v_n\}$ by iterating for a constant $K\geq 2\lambda$,
\begin{equation}\label{4.5}
	(\Delta -K)v_n=\lambda e^{u_0+v_{n-1}}(e^{u_0+v_{n-1}}-1)-Kv_{n-1}+\frac{4\pi M}{\vert V\vert}.
\end{equation}

We next prove that $\{v_n\}$ is a monotone sequence and it converges to a solution of equation \eqref{4.2}.

%\end{proof}

\begin{lemma}\label{lemma4.2}
	Let $\{v_n\}$  be a sequence defined by $\eqref{4.5}$. Then $$v_0\geq v_1 \geq v_2\geq... \geq v_n...\geq v_+ $$ for any upper solution $v_+$ of $\eqref{4.2}$ .
\end{lemma}

\begin{proof}
	We prove the Lemma by induction. As $v_0=-u_0$, for $v_1$ we have by $\eqref{4.5}$,
	\begin{equation*}
		(\Delta -K)v_1=Ku_0+\frac{4\pi M}{|V|}.
	\end{equation*}
	Together with $\eqref{4.1}$, we obtain $$(\Delta -K)(v_1-v_0)(x)=4\pi \sum_{j=1}^{M}{\delta_{p_j}(x)}\geq 0$$for any $x\in V$, and $$K\int_{V}{(v_1-v_0)}=-4\pi M<0.$$
	Therefore $v_1-v_0\leq 0$ by Lemma \ref{lemma4.1}. Suppose that $v_0 \geq v_1\geq ...\geq v_k$ for $k\ge 1$. From $\eqref{4.5}$
	and $K\geq 2\lambda$, we get
	\[
	\begin{split}
	(\Delta -K)(v_{k+1}-v_k) &= \lambda e^{2u_0 +2v_k}-\lambda e^{u_0+v_k}-Kv_k-\lambda e^{2u_0+2v_{k-1}}+\lambda e^{u_0+v_{k-1}}+Kv_k\\
	&= \lambda e^{2u_0}(e^{2v_k}-e^{2v_{k-1}})-\lambda e^{u_0}(e^{v_k}-e^{v_{k-1}})-K(v_k-v_{k-1})\\
	&\geq \lambda e^{2u_0}(e^{2v_k}-e^{2v_{k-1}})-K(v_k-v_{k-1})\\
	&=2\lambda e^{2u_0+2v^*}(v_k-v_{k-1})-K(v_k-v_{k-1})\\
	&\geq K(e^{2u_0+2v_0}-1)(v_k-v_{k-1})\\
	&\geq 0.
	\end{split}
	\]
	Where $v_k\leq v^*\leq v_{k-1}\leq v_0.$ Lemma 4.1 then implies that $v_{k+1}-v_k\leq 0$ on $V$.
	
	Next we prove that $v_k\geq v_+$ for any $k$.\\
	First consider the case $k=0$. From $\eqref{4.1}$ and $\eqref{4.4}$ ,

	\begin{equation}\label{d}
		\begin{split}
		\Delta (v_+-v_0) &\geq \lambda e^{u_0 +v_+}(e^{u_0+v_+}-1)+4\pi \sum_{j=1}^{M}{\delta_{p_j}}\\
		&\geq \lambda e^{u_0 +v_+}(e^{u_0+v_+}-1)\\
		&= \lambda e^{v_+-v_0}(e^{v_+-v_0}-1).
		\end{split}
	    %\notag	
	\end{equation}
Let $v_+(x_0)-v_0(x_0)=\max_{x\in V}{\{v_+(x_0)-v_0(x_0)\}}$. We only need to prove that $v_+(x_0)-v_0(x_0)\leq 0$. Suppose not, then from \eqref{d} we have $$\Delta (v_+-v_0)(x_0)>0.$$
which contradicts with the assumption that $x_0$ is a point where $v_+-v_0$ attains maximum in $V$. Hence $v_+-v_0\leq 0$ in $V$.	
Now suppose that $v_+\leq u_k$ for $k\ge 0$. From \eqref{4.4} and \eqref{4.5}, we have
	\[
	\begin{split}
	    (\Delta -K) (v_+-v_{k+1}) &= \lambda e^{2u_0}(e^{2v_+}-e^{2v_k})-K(v_+-v_k)-\lambda e^{u_0}(e^{v_+}-e^{v_k})\\
      	&\geq \lambda e^{2u_0}(e^{2v_+}-e^{2v_k})-K(v_+-v_k)\\
	    &= 2\lambda e^{2u_0+2v^*}(v_+-v_k)-K(v_+-v_k)\\
	    &\geq K(e^{2u_0+2v_0}-1)(v_+-v_k)\\
  	    &=0,
	\end{split}
	\]
where $v_+ \leq v^* \leq v_k \leq v_0$. So Lemma \ref{lemma4.1} implies that $v_{k+1}\geq v_+$.
	
This finishes the proof of Lemma \ref{lemma4.2}.

\end{proof}

\begin{lemma}
The equation \eqref{equ:4} has a solution on $G$, when $\lambda$ is sufficiently big.
\end{lemma}

\begin{proof}
	We only need to prove that equation \eqref{4.2} has an upper solution $v_+$. Suppose $u_0$ is a solution of \eqref{4.1}. Choose $v_+ = -c'' <0$ to be a constant function, where $-c''$ is sufficiently small such that $u_0 +v_+ <0$ in $V$. Then $e^{u_0+v_+}-1 <0$. So we can choose $\lambda >0$ big enough such that $$\lambda e^{u_0+v_+}(e^{u_0+v_+}-1)+\frac{4\pi M}{|V|}<0. $$
	Therefore $$0=\Delta v{_+}>\lambda e^{u_0+v_+}(e^{u_0+v_+}-1)+\frac{4\pi M}{|V|}.$$
	So $v_+\equiv -c$ is an upper solution of \eqref{4.2}.
\end{proof}

\begin{lemma}
If $u$ is a solution of equation \eqref{equ:4}  on $G$, then $u<0$ on $G$.
\end{lemma}

\begin{proof}
Let $u(x_0)=\max_{x\in V}{\{u(x)\}}$, we only need to show that $u(x_0)<0$. Suppose $u(x_0)\ge 0$. Then $e^{u(x_0)}-1\ge 0$. From equation \eqref{equ:4} we get
that
$$\Delta u(x_0)\geq 0,$$

 that is
 $$\sum_{y \sim x}{u(y)}\geq d_{x_0}u(x_0).$$
  This implies that for any
  $$y\sim x,u(y)\geq u(x_0).$$
Since $G$ is a connected finite graph, by iterating the above process, we get that for any $$ y\in V,u(y)=u(x_0).$$

So the left side of equation \eqref{equ:4} is $0$ and the right side is positive on $p_j\in V$, which is a contradiction.
\end{proof}

Now we prove Theorem \ref{2.2}, which is similar to the proof of Lemma 4 in \cite{CY95}.

\begin{proof}
Denote $$\Lambda=\{\lambda>0| \lambda\,  \mbox{is such that equation \eqref{equ:4} has a solution} \}.$$
We will show that $\Lambda$ is an interval. Suppose that $\lambda'\in \Lambda$. We need to prove that
$$[\lambda', +\infty)\in \Lambda.$$
In fact, let $u'=u_{0}+v'$ is the solution of equation \eqref{equ:4} at $\lambda=\lambda'$, where $v'$ is the corresponding
solution of equation \eqref{4.2}. Since $$u'=u_{0}+v'<0,$$ we see that $v'$ is an upper solution of equation \eqref{4.2} for any  $\lambda\geq\lambda'$.
By Lemma \ref{lemma4.2}, we obtain that $\lambda\in \Lambda$ as desired.

Set $ \lambda_c=\inf \left\{ \lambda| \lambda\in\Lambda\right\}.$ Then $\lambda\ge\frac{16\pi M}{|V|}$ for any $\lambda> \lambda_c$ by \eqref{4.3} and that $\Lambda$ is an interval. Taking the limit, we get
that
$$\lambda_c\ge \frac{16\pi N}{|V|}.$$
\end{proof}

{\bf Acknowledgements.} Y. Lin is supported by the National Science Foundation of China (Grant No. 11271011), S.-T. Yau is supported by the NSF DMS-0804. Part of the work was done when Y. Lin visited the Harvard CMSA in 2018.

\bibliographystyle{amsalpha}

An Huang,\\
Department of Mathematics, Brandies University, Waltham, Massachusetts, USA\\
\textsf{anhuang@brandeis.edu}\\
Yong Lin,\\
Yau Mathematical Sciences Center, Tsinghua University, Beijing, China\\
\textsf{yonglin@tsinghua.edu.cn}\\
Shing-Tung Yau\\
Department of Mathematics, Harvard University, Cambridge, Massachusetts, USA\\
\textsf{yau@math.harvard.edu}\\

\end{document}